\documentclass[12pt]{amsart}
\usepackage{hyperref,amsmath,amsfonts,amssymb,amsthm,enumerate,mathtools,graphicx}
\usepackage[a4paper, left=2cm, right=2cm, top=3cm, bottom=2cm]{geometry}
\usepackage[utf8]{inputenc}
\usepackage{mathtools}
\usepackage{enumitem}
\usepackage{float}

\theoremstyle{plain}
\newtheorem{thm}{Theorem}
\newtheorem{lemma}{Lemma}

\newtheorem{prop}{Proposition}
\newtheorem{conj}{Conjecture}
\newtheorem{df}{Definition}
\newtheorem*{notat}{Notation}

\theoremstyle{remark}
\newtheorem*{rem}{Remark}

\newcommand{\R}{\mathbb{R}}
\newcommand{\N}{\mathbb{N}}
\newcommand{\Es}[1]{\mathbb{S}^{#1}}

\newcommand*{\diam}{\mathop{}\mathrm{diam\,}}

\newcommand{\vect}[1]{\boldsymbol{#1}}

\newcommand{\abs}[1]{\left|{#1}\right|}

\newcommand{\Abs}[1]{\left\lVert{#1}\right\rVert}
\newcommand{\dist}[2]{\mathrm{dist}\left({#1},{#2}\right)}

\title{A note on directional Lipschitz continuity in~the~Euclidean plane}

\author[D. Hru\v{s}ka]{David Hru\v{s}ka}
\address{Leipzig University, Faculty of Mathematics and Computer Science,
Institute of Mathematics,
Augustusplatz 10, 04 109 Leipzig, Germany}
\email{hruska@math.uni-leipzig.de}

\begin{document}

\begin{abstract}
    We prove a stronger version of a conjecture stated in a paper from 2017 by J. M. Ash and S. Catoiu concerning relations between various notions of the Lipschitz property and differentiability in the Euclidean plane. We also provide an improved version of the main result of that paper. 
\end{abstract}

\maketitle
\section{Introduction and main results}

In the conclusion of \cite{ash} (Conjecture 7 therein) the authors suggest the following might hold (see the next section for the definitions):
\begin{conj}\label{domn}
Let $f:\R^2\to\R$ and $E=\bigl\{\vect{x}\in\R^2: f\text{~is Lipschitz at~}\vect{x} \text{~in every direction}\bigr\}$.
Then $$F=\bigl\{\vect{x}\in E: f\text{~is not Lipschitz at~}\vect{x} \text{~relative to~}E\bigr\}$$ is a null set.
\end{conj} 
\noindent 

In order to put this statement in context, consider the following heuristic principle (cf. also the discussion concluding the first section of \cite{ash}):
\begin{align}\label{heuristics}
\begin{split}
    &\text{If $f:\R^d\to\R$ is any function and $E$ denotes the set of points where $f$ is \emph{smooth}}~\text{in} \\&\text{\emph{many directions}, then $f$ is Lipschitz relative to $E$ on a~\emph{large subset} of $E$ (or $\R^d$).}
\end{split}
    \end{align}
The main result of \cite{ash}, Theorem 4, disproves a variant of this principle with $d=2$ and the emphasized parts replaced by \emph{differentiable}, \emph{almost every direction} and \emph{a subset of $E$ of full measure}, respectively. Conjecture \ref{domn} then concerns with another precise version of \eqref{heuristics}.

\medskip
This paper contributes to investigation of the principle \eqref{heuristics} in two ways. On one hand, we confirm Conjecture \ref{domn} in a stronger form. Namely, we prove 
\begin{thm}\label{thm1}
Let $f:\R^2\to\R$ and define $$E=\bigl\{\vect{x}\in\R^2: f\text{~is Lipschitz at~}\vect{x} \text{~in residual many directions}\bigr\}.$$ 
Then the set
$$F=\bigl\{\vect{x}\in E: f\text{~is not Lipschitz at~} \vect{x} \text{~relative to~}E\bigr\}$$ is $\sigma$-porous.
\end{thm} 

On the other hand, we have
\begin{thm}\label{thm2}
There exist sets $Z\subset\R^2$ and $E\subset\R^2$ such that $\abs{\R^2\setminus E}=0$ and $\chi_Z$ (the characteristic function of $Z$) is at each point of $E$ differentiable in almost every direction, but $\chi_Z$ is nowhere Lipschitz relative to $E$.
\end{thm} 
This construction refines Theorem 4 from \cite{ash} and, consequently, further restricts the extent to which the principle \eqref{heuristics} can possibly be valid. 

\subsection{Preliminaries}

Let us start with some useful notation conventions and necessary definitions.

\begin{notat}
The Euclidean norm on $\R^d$ will be denoted by $\Abs{\cdot}$ and the corresponding unit sphere by $\Es{d-1}$. For $\vect{x}\in\R^d$ and $r>0$ we write $B(\vect{x},r)$ for the open ball centered at $\vect{x}$ with radius $r$. For $k>0$ we define $k*B(\vect{x},r)=B(\vect{x},kr)$. When $A\subset X$ is measurable with respect to an implicitly given measure $\mu$, we write $\abs{A}=\mu(A)$. If $\mathcal{A}$ is a countable system of sets we write 
\begin{equation*}
    \bigcup \mathcal{A}=\{x:x\in A \text{~for some~}A\in\mathcal{A}\}\text{~and~}\limsup \mathcal{A}=\{x:x\in A \text{~for infinitely many~}A\in\mathcal{A}\}.
\end{equation*}
\end{notat}

\begin{df}[$\sigma$-porosity in $\R^d$]
We say that a set $A\subset\R^d$ is \emph{porous at} $\vect{x}\in\R^d$, if there is $0<p<1$ such that for every $\varepsilon>0$ there exists $\vect{y}\in B(\vect{x},\varepsilon)$ such that $\vect{x}\neq\vect{y}$ and $B(\vect{y},p\!\Abs{\vect{x}-\vect{y}})\cap A=\emptyset$. If $A$ is porous at $\vect{x}$ for all $\vect{x}\in A$ we say that $A$ is \emph{porous}. A subset of $\R^d$ is called \emph{$\sigma$-porous} if it is a countable union of porous sets.
\end{df}

\begin{df}[Directional and relative Lipschitz property]\label{dir_lip}
Let $G\subset\R^d$ be an open set, ${f:G\to\R}$ a function and $\vect{u}\in \mathbb{S}^{d-1}$. We say that $f$ is \emph{Lipschitz at $\vect{x}\in G$ in direction $\vect{u}$} if  $$\limsup_{h\to0+} \frac{\abs{f(\vect{x}+h\vect{u})-f(\vect{x})}}{h}<\infty.$$ 

\noindent We say that $f$ is \emph{Lipschitz} at $\vect{x}$ \emph{in residual many directions} (resp. \emph{in almost every direction}) if the set $$\left\{u\in\Es{d-1}:f \text{~is not Lipschitz at~}\vect{x}\text{~in direction~}\vect{u}\right\}$$ is of the first Baire category (resp. of zero Hausdorff measure $\mathcal{H}^{d-1}$) in $\Es{d-1}$.

\medskip
\noindent Differentiability at $\vect{x}\in G$ in direction $\vect{u}\in \Es{d-1}$ and in almost every direction is defined analogously.

\medskip
\noindent We say that $f$ is \emph{Lipschitz} at $\vect{x}\in G$ \emph{relative to} $M\subset G$ if $\vect{x}$ is not a limit point of $M$ or 
$$\limsup_{\vect{y}\to\vect{x},\,\vect{y}\in M}\frac{\abs{f(\vect{y})-f(\vect{x})}}{\Abs{\vect{y}-\vect{x}}}<\infty.$$ 
\end{df}

\begin{df}[Circular sector]\label{kuzel}
For  $d\geq 2$, $\vect{x}\in\R^d$, $\vect{u}\in\mathbb{S}^{d-1}$ and $\delta,r>0$ we define an \emph{open circular sector} with vertex $\vect{x}$, axis $\vect{u}$, opening parameter $\delta$ and radius $r$ by $$C(\vect{x},\vect{u},\delta,r)=\left\{\vect{x}\neq\vect{y}\in\R^d:\Abs{\frac{\vect{y}-\vect{x}}{\Abs{\vect{y}-\vect{x}}}-\vect{u}}<\delta\right\}\cap B(\vect{x},r).$$ 
\end{df}

\section{Proof of the positive result}

The proof of Theorem \ref{thm1} relies on the following variant of Proposition 3.1. from \cite{zajicek}.

\begin{lemma}\label{zajicek_prop}
Let $d\geq2$, $f:\R^d\to\R$ be any function and $M\subset \R^d$ any set. Let $A$ be the set of all points $\vect{x} \in M$ for which there exists an open circular sector $C=C(\vect{x},\vect{u},\delta,r)$ such that $f$ is Lipschitz at $\vect{x}$ relative to $M\cap C$ but $f$ is not Lipschitz at $\vect{x}$ relative to $M$. Then $A$ is $\sigma$-porous.
\end{lemma}

Before we provide the proof (a slight modification of that given in \cite{zajicek}), let us recall the following elementary fact (see e.g. Lemma~5.1 in \cite{fact}).

\begin{prop}\label{triangle}
If $\vect{u}$, $\vect{v}\in\R^d\setminus\{\vect{0}\}$, then $\Abs{\frac{\vect{u}}{\Abs{\vect{u}}}-\frac{\vect{v}}{\Abs{\vect{v}}}}\leq \frac{2}{\Abs{\vect{v}}}\Abs{\vect{u}-\vect{v}}$.
\end{prop}  

\begin{proof}[Proof of Lemma \ref{zajicek_prop}]
    Let $\{\vect{v}_n: n \in \N\}$ be a dense subset of $\Es{d-1}$. For positive integers $k$, $p$, $n$,~$m$ we denote
     by $A_{k,p,n,m}$ the set of all points $\vect{x} \in A$ such that
     \begin{equation}\label{trij}
     \frac{\abs{f(\vect{y})-f(\vect{x})}}{\Abs{\vect{y}-\vect{x}}} < k\ \ \text{whenever}\ \ 0< \Abs{\vect{y}-\vect{x}}< \frac{1}{p},\ \vect{y}\in M \ \text{and}\ \ 
      \left\| \frac{\vect y-\vect x}{\Abs{\vect{y}-\vect{x}}}- \vect{v}_n\right\| < \frac{1}{m}.
     \end{equation}
     Since $A$ is clearly the (countable) union of all sets $A_{k,p,n,m}$, it is sufficient to prove that for given $k$, $p$, $n$, $m\in\N$  the set
    $A_{k,p,n,m}$ is porous at any $\vect{x} \in A_{k,p,n,m}$. To this end, find a sequence $\vect{y}_i \xrightarrow{i\to\infty} \vect{x}$ such that for each
      $i \in \N$ we have $\vect{y}_i\neq\vect{x}$, $\vect{y}_i\in M$ and 
     \begin{equation}\label{star}
      \frac{\abs{f(\vect{y}_i)-f(\vect{x})}}{\Abs{\vect{y}_i-\vect{x}}} > k(12 m +4)
      \end{equation} 
       Set  $r_i= \Abs{\vect{y}_i-\vect{x}}$ and $\vect{x}_i= \vect{x}- 6mr_i\vect{v}_n$. It is sufficient to prove that there exists
     $i_0 \in \N$ such that
     \begin{equation}\label{prmn}
      B(\vect{x}_i,r_i) \cap A_{k,p,n,m} = \emptyset\ \ \ \text{for each}\ \ \ i \geq i_0.
     \end{equation}
     Consider $\vect{z}_{i} \in  B(\vect{x}_i,r_i) \cap A_{k,p,n,m}$ for some $i\in\N$ and observe that
     \begin{equation*}
     6mr_i -r_i \leq \Abs{\vect{x}-\vect{x}_i}- \Abs{\vect{x}_i-\vect{z}_i} \leq \Abs{\vect{x}-\vect{z}_i} \leq \Abs{\vect{x}-\vect{x}_i}+ \Abs{\vect{x}_i-\vect{z}_i} \leq 6mr_i +r_i
\end{equation*} and thus
\begin{equation*}
    6mr_i -2r_i \leq\Abs{\vect{x}-\vect{z}_i} -r_i \leq \Abs{\vect{y}_i-\vect{z}_i}  \leq\Abs{\vect{x}-\vect{z}_i}+r_i\leq 6mr_i +2r_i.
\end{equation*}
     These inequalities imply that there exists $i_0 \in \N$, which we fix, such that
     \begin{equation}\label{jlm}
     0<  \Abs{\vect{x}-\vect{z}_i} < \frac{1}{p}\ \ \text{and}\ \ 0<  \Abs{\vect{y}_i-\vect{z}_i} < \frac{1}{p}\quad \text{for every}\ \  i \geq i_0.
     \end{equation}
     Proposition \ref{triangle} yields
     \begin{equation}\label{nabla}
          \left\| \frac{\vect{x}-\vect{z}_i}{\|\vect{x}-\vect{z}_i\|}- \vect{v}_n\right\|= \left\| \frac{\vect{x}-\vect{z}_i}{\|\vect{x}-\vect{z}_i\|}-\frac{\vect{x}-\vect{x}_i}{\|\vect{x}-\vect{x}_i\|}\right\|
      \leq 2\frac{\|\vect{x}_i-\vect{z}_i\|}{\|\vect{x}-\vect{x}_i\|} 
      < \frac{2 r_i}{6mr_i} < \frac{1}{m} 
          \end{equation}
     and analogously
     \begin{equation}\label{ctyr}
     \left\| \frac{\vect{y}_i-\vect{z}_i}{\|\vect{y}_i-\vect{z}_i\|}- \vect{v}_n\right\| = \left\| \frac{\vect{y}_i-\vect{z}_i}{\|\vect{y}_i-\vect{z}_i\|}-\frac{\vect{x}-\vect{x}_i}{\|\vect{x}-\vect{x}_i\|}\right\| \leq 2 \frac{\|\vect{y}_i-\vect{x}\| +\|\vect{x}_i-\vect{z}_i\|}{\|\vect{x}-\vect{x}_i\|}
      < \frac{4r_i}{6mr_i} < \frac{1}{m}.
     \end{equation}
     Since $\vect{x}$, $\vect{y}_i\in M$, conditions \eqref{trij}, \eqref{jlm}, \eqref{nabla} and \eqref{ctyr} imply that if $i \geq i_0$, then
     $$ \frac{\abs{f(\vect{x})-f(\vect{z}_i)}}{\|\vect{x}-\vect{z}_i\|} < k\ \ \ \text{and}\ \ \  \frac{\abs{f(\vect{y}_i)-f(\vect{z}_i)}}{\|\vect{y}_i-\vect{z}_i\|} < k.$$
     Using also inequality \eqref{star}, we obtain for every $i \geq i_0$ that
     \begin{align*}
     r_i k (12m+4) &<  \abs{f(\vect{x})-f(\vect{y}_i)} \leq      \abs{f(\vect{x})-f(\vect{z}_i)} + \abs{f(\vect{y}_i)-f(\vect{z}_i)} \\
     &< k \|\vect{x}-\vect{z}_i\| + k \|\vect{y}_i-\vect{z}_i\| \leq k (6mr_i +r_i +6mr_i +2r_i),
         \end{align*}
     which is impossible. We have thus proved \eqref{prmn}.
 \end{proof}
 
\noindent The proof of Theorem~\ref{thm1} now boils down to fulfilling the assumptions of Lemma~\ref{zajicek_prop}.

\begin{proof}[Proof of Theorem \ref{thm1}]
Choose $\vect{x}\in E$ arbitrarily and define $$S_n=\left\{\vect{u}\in\Es{1}:\frac{\abs{f(\vect{x}+h\vect{u})-f(\vect{x})}}{h}<n\text{~for all~}0<h<\frac{1}{n}\right\}.$$ The definition of $E$ then yields $\Es{1}=S_0\cup\bigcup_{n=1}^{\infty} S_n$ where $S_0\subset \Es{1}$ is of the first category in $\Es{1}$ (which is a complete metric space). The Baire category theorem shows
 that there exists $n\in\N$ such that $S_n$ is dense in a nonempty open subset of $\Es{1}$. Hence there exists $\vect{v}\in\Es{1}$ and $\delta>0$ such that
 \begin{equation}\label{husto_v_kuzelu_2D}
     \{\vect{u}\in\Es{1}:\Abs{\vect{u}-\vect{v}}<\delta \}\subset \overline{S_n}.
 \end{equation} In other words, using Definition \ref{kuzel}, the union of open segments $U=\bigcup_{\vect{u}\in S_n}\left\{\vect{x}+t\vect{u}: t\in\left(0,\frac{1}{n}\right)\right\}$  is dense in $C=C\left(\vect{x},\vect{v},\delta,\frac{1}{n}\right)$.
 
 Consider any $\vect{y}\in C\cap E$. Due to the definition of $E$ the set of directions $\vect{u}\in\Es{1}$ such that $f$ is Lipschitz at $\vect{y}$ in direction $\vect{u}$ is residual in $\Es{1}$. Let us pick such $\vect{u}$ which is moreover different from $\pm\frac{\vect{y}-\vect{x}}{\Abs{\vect{y}-\vect{x}}}$ and observe that, since every two non-parallel lines in $\R^2$ intersect each other, it follows from condition \eqref{husto_v_kuzelu_2D} that there is a sequence $\left(t_i\right)_{i=1}^{\infty}$ of positive numbers such that $\vect{y}+t_i\vect{u}\in U$ for all $i\in\N$ and $t_i \xrightarrow{i\to\infty}0+$. Since $f$ is continuous at $\vect{y}$ in direction $\vect{u}$ it follows that $$\frac{\abs{f(\vect{y})-f(\vect{x})}}{\Abs{\vect{y}-\vect{x}}}=\lim_{i\to\infty}\frac{\abs{f(\vect{y}+t_i\vect{u})-f(\vect{x})}}{\Abs{\vect{y}+t_i\vect{u}-\vect{x}}}\leq 
 n.$$ We have proved that for every $\vect{x}\in E$ there is an open circular sector $C$ with vertex at $\vect{x}$ such that $f$ is Lipschitz at $\vect{x}$ relative to $E\cap C$. To finish the proof, apply Lemma \ref{zajicek_prop} with $M = E$ from which it follows that $A = F$.
\end{proof}

\begin{rem}
Since Theorem 4 from \cite{ash} does not essentially depend on dimension $d\geq 2$ it is natural to investigate the principle \eqref{heuristics} in higher dimensions. However, although Lemma \ref{zajicek_prop} works for all $d\geq 2$, our proof of Theorem \ref{thm1} makes use of a specific property of the Euclidean plane, namely that every two lines are either parallel or have a common point, and hence this proof cannot be directly generalized to higher dimensions. It is not known to the author whether such a  generalization of Theorem \ref{thm1} holds without further assumptions on $f$. 
\end{rem}

\section{Proof of the negative result}

Our construction further exploits the ideas introduced in the proof of Theorem 4 from \cite{ash} and, in addition, employs an iterative argument. We start by stating two useful facts proven in \cite{ash}.

\begin{prop}\label{uhel_kruhu}
Let $B=B(\vect{x},r)$ and $\vect{y}\in\R^2\setminus \overline{B}$. Then the angle $\theta$ that $B$ subtends when viewed from $\vect{y}$ satisfies $\theta<\pi\frac{r}{\Abs{\vect{y}-\vect{x}}}$.
\end{prop}

\begin{prop}[Borel-Cantelli lemma] \label{bc} If $\mathcal{A}=\left\{A_i:i\in\N\right\}$ is a system of measurable sets satisfying $\sum_{i=1}^{\infty}\abs{A_i}<\infty$, then $\abs{\limsup\mathcal{A}}=0$.
\end{prop}

\smallskip
\begin{proof}[Proof of Theorem \ref{thm2}]
We proceed in three steps. The first step generalizes the construction from Paragraph 3.2 of \cite{ash} to arbitrary open subset of $\R^2$ and defines the set $Z$ by iterating this construction. In the second step we define the ``good" set $E$ and prove that it has full measure in $\R^2$. The third step shows that $\chi_Z$ is nowhere Lipschitz (even continuous) relative to $E$.

\medskip
\noindent \textit{Step 1.}
Given any non-empty open set $G\subset\R^2$, let $\left\{\vect{x}_i: i\in\N\right\}$ be a dense subset of $G$. Set $n_1=1$ and take an open ball $B_1=B(\vect{x}_1,r_1)$ such that $B_1\subset G$ and 
\begin{equation}\label{prvni_polomer}
    r_1\leq\min\left\{\frac{1}{4},\frac{\diam G}{8}\right\}.
\end{equation} For $k\in\N$ and pairwise disjoint open discs $B_1,\dots,B_k$ let $n_{k+1}$ be the smallest $n\in\N$ such that $\vect{x}_n\notin\overline{\bigcup_{i=1}^k B_i}$, define
\begin{equation}\label{polomery_2}
    r_{k+1}=\min\left\{\frac{\min\left\{1,\frac{1}{2}\diam{G}\right\}}{4^{k+1}},\frac{1}{2}\dist{\vect{x}_{n_{k+1}}}{\left(\R^2\setminus G\right)\cup\bigcup_{i=1}^{k}B_i}\right\}
\end{equation} and finally let $B_{k+1}=B(\vect{x}_{n_{k+1}},r_{k+1})$. It is easy to see that this inductive procedure continues beyond any finite number of steps ($n_{k+1}$ always exists and $r_{k+1}>0$ is well defined) and yields a~countable system 
\begin{equation*}
\mathcal{S}(G)=\left\{B_i:i\in\N\right\}    
\end{equation*} of pairwise disjoint open discs contained in $G$ such that its union is dense in $G$. Next, we define the corresponding system of enlarged closed discs 
\begin{equation*}
    \mathcal{E}(G)=\left\{\overline{2^i*B_i}:i\in\N\right\}.
\end{equation*} 
Note that it follows from conditions \eqref{prvni_polomer} and \eqref{polomery_2} that $2^ir_i\leq \frac{1}{2^{i}}$ for all $i\in\N$ and any non-empty open set $G\subset \R^2$ thus satisfies
 \begin{equation}\label{finite_measure}
      \sum_{B\in \mathcal{S}(G)}\abs{B}\leq\sum_{B\in \mathcal{E}(G)}\abs{B}\leq\pi\sum_{i=1}^{\infty}\left(\frac{1}{2^i}\right)^2 < \infty.
 \end{equation} 
For any disc $B$ of radius $R>0$ conditions \eqref{prvni_polomer} and \eqref{polomery_2} similarly imply
  \begin{equation}\label{small_measure}
     \abs{\bigcup\mathcal{S}(B)}\leq \pi R^2\sum_{i=1}^{\infty}\left(\frac{1}{4^i}\right)^2<\frac{1}{2}\abs{B}
\end{equation} which in particular shows that
\begin{equation}\label{small_measure_cor}
    B\setminus\bigcup\mathcal{S}(B)\text{~is a Borel set of positive measure}.
\end{equation}
Set $\mathcal{D}_1=\mathcal{S}(\R^2)$ and for $k\in\N$ define inductively $\mathcal{D}_{k+1}= \bigcup_{B\in\mathcal{D}_k}\mathcal{S}(B)$. Finally, let us set
\begin{equation}\label{def_Z}
    Z=\bigcup_{m=1}^{\infty} \left(\bigcup\mathcal{D}_{2m-1}\setminus\bigcup\mathcal{D}_{2m}\right).
\end{equation} In terms of $\chi_Z$, this construction starts with the zero function, changes its value to one on discs forming system $\mathcal{D}_1$, then changes it back to zero on all the subsystems forming $\mathcal{D}_2$, etc. 

\medskip
\noindent \textit{Step 2.} 
First define 
\begin{equation*}
    E=\left\{\vect{x}\in\R^2:\chi_Z\text{~is differentiable in almost every direction at~}\vect{x}\right\}.
\end{equation*} Note that the set $C=\bigcup_{k\in\N}\bigcup_{B\in \mathcal{D}_k} \partial B$ has zero Lebesgue measure. Using Proposition \ref{bc} it follows from \eqref{finite_measure} that 
\begin{equation}\label{n_nula}
    N_0=C\cup\limsup\mathcal{E}(\R^2)
\end{equation} is also a null set. 

Let us denote the angle subtended by a disc $D$ when viewed from point $\vect{y}\notin\overline{D}$ by $\theta_D$ and~take any $\vect{y}\in\R^2\setminus\left(\bigcup\mathcal{D}_1\cup N_0\right)$. Definition \eqref{n_nula} implies that $\displaystyle \vect{y}\notin \overline{B}$ whenever $B\in\mathcal{D}_1$ and if $B=2^j*B_j\in\mathcal{E}(\R^2)$, then $\vect{y}\in\overline{B}$ only if $j\in F$ for some finite set $F\subset\N$. Proposition \ref{uhel_kruhu} yields
\begin{equation*}
    \sum_{D\in\mathcal{D}_1\setminus\left\{B_j:j\in F\right\}}\theta_D\leq \pi\sum_{i=1}^{\infty}\frac{1}{2^i}<\infty
\end{equation*} and Proposition \ref{bc} (now for the Hausdorff measure $\mathcal{H}^1$ on $\Es{1}$) then shows that the set of directions $\vect{u}\in\Es{1}$, such that the ray $\left\{\vect{y}+t\vect{u}:t\in(0,\infty)\right\}$ intersects infinitely many discs from $\mathcal{D}_1$, is null in $\Es{1}$. It follows that in all the other directions $\chi_Z$ is identically zero in a positive distance from $\vect{y}$ and we conclude that $\R^2\setminus\left(\bigcup\mathcal{D}_1\cup N_0\right)\subset E$.

\smallskip
For any $m\in\N$ we define
\begin{equation*}
    N_m=C\cup\bigcup_{B\in\mathcal{D}_m}\left(B\cap\limsup\mathcal{E}(B)\right).
\end{equation*} Since $\mathcal{D}_m$ consists of pairwise disjoint open discs and set $N_m$ has zero measure, analogous arguments as above show that for every $\vect{y}\in\left(\bigcup \mathcal{D}_m\right)\setminus (\bigcup \mathcal{D}_{m+1}\cup N_m)$ the function $\chi_Z$ is constant in almost every direction in a positive distance from $\vect{y}$. Hence ${\left(\bigcup \mathcal{D}_m\right)\setminus (\bigcup \mathcal{D}_{m+1}\cup N_m)\subset E}$ for every $m\in\N$ and altogether we end up with
\begin{equation*}
    \R^2\setminus E \subset  \left(\bigcup_{i=0}^{\infty}N_i\right)\,\cup\,\left(\bigcap_{m=1}^{\infty}\bigcup\mathcal{D}_m\right).
\end{equation*} Conditions \eqref{finite_measure} and \eqref{small_measure} together with the definition of $\mathcal{D}_m$ imply $\abs{\bigcup \mathcal{D}_m}\xrightarrow{m\to\infty}0$ and we conclude that $\abs{\R^2\setminus E}=0$.

\medskip
\noindent \textit{Step 3.} 
Observe that since $E\subset \R^2$ has full measure, it follows from assertion \eqref{small_measure_cor} that for every $m\in\N$ and $B\in\mathcal{D}_m$ the set $\left(B\setminus\bigcup\mathcal{S}(B)\right)\cap E$ is non-empty. In particular, it follows that 
\begin{equation}\label{observace}
   \text{~for any odd~} m\in\N \text{~and disc~} B\in\mathcal{D}_m, \text{~the equality~} \chi_Z=1 \text{~holds somewhere in~} E\cap B.
\end{equation} 

\smallskip
\noindent Take any $\vect{x}\in\R^2\setminus \bigcup \mathcal{D}_1$  and $\varepsilon>0$. Then $\chi_Z(\vect{x})=0$ and note that it follows from condition \eqref{polomery_2} that the pairwise disjoint discs forming $\mathcal{D}_1$ do not touch (i.e. their boundary circles are pairwise disjoint as well). Thus there is at most one disc $D\in\mathcal{D}_1$ such that $\vect{x}\in\partial D$. For the same reason, union of any finite subsystem $\mathcal{D}'\subset\mathcal{D}_1$ not containing $D$ has a positive distance from $ D\cap B(\vect{x},\frac{\varepsilon}{2})$ (or simply to $\{\vect{x}\}$ when there is no such $D$). In both cases the density of $\bigcup\mathcal{D}_1$ in $\R^2$ implies that there must be infinitely many discs from $\mathcal{D}_1$ intersecting $B(\vect{x},\frac{\varepsilon}{2})$ and due to the decay of their radii implied by formula \eqref{polomery_2} we can find one contained in $B(\vect{x},\varepsilon)$. Assertion \eqref{observace} then yields a point $\vect{y}\in B(\vect{x},\varepsilon)\cap E$ such that $\chi_Z(\vect{y})=1$. Since $\varepsilon>0$ was arbitrary, we conclude that $\chi_Z$ is not continuous at $\vect{x}$ relative to $E$.

Analogous arguments work for any $m\in\N$, $B\in\mathcal{D}_m$ and $\vect{x}\in B\setminus \bigcup\mathcal{D}_{m+1}=B\setminus \bigcup\mathcal{S}(B)$ (recall that elements of $\mathcal{D}_m$ are pairwise disjoint) with the only modification being that for odd $m$ the roles of values $0$ and $1$ of $\chi_Z$ interchange.

Finally, for any $\vect{x}\in\bigcap_{m=1}^{\infty}\bigcup\mathcal{D}_m$ and $\varepsilon>0$ observe that $\chi_Z(\vect{x})=0$ and conditions \eqref{prvni_polomer} and \eqref{polomery_2} imply $\sup_{B(\vect{z},r)\in\mathcal{D}_k}r\xrightarrow{k\to\infty}0$. Hence there are index $i\in\N$ and disc $D\in\mathcal{D}_{2i-1}$ such that $\vect{x}\in D\subset B(\vect{x},\varepsilon)$ and again by observation \eqref{observace} we find a point $\vect{y}\in D\cap E$ such that $\chi_Z(\vect{y})=1$. As $\varepsilon>0$ was arbitrary, the conclusion follows.  
\end{proof}

\end{document}